\documentclass[reqno]{amsart}

\usepackage{amsmath}
\usepackage{amsfonts}
\usepackage{amssymb,enumerate}
\usepackage{amsthm}
\usepackage[all]{xy}
\usepackage{rotating}
\usepackage{hyperref}
\usepackage{color}

\theoremstyle{plain}
\newtheorem{lem}{Lemma}[section]
\newtheorem{cor}[lem]{Corollary}
\newtheorem{prop}[lem]{Proposition}
\newtheorem{thm}[lem]{Theorem}

\newtheorem*{mthm*}{Main Theorem}

\theoremstyle{definition}
\newtheorem{defn}[lem]{Definition}

\newtheorem{para}[lem]{}

\newtheorem*{convention*}{Convention}
\newtheorem*{ARC}{Auslander-Reiten Conjecture}
\newtheorem*{NLC}{Na\"{\i}ve Lifting Conjecture}

\newcommand{\pd}{\operatorname{pd}}

\newcommand{\id}{\operatorname{id}}

\newcommand{\lotimes}{\otimes^{\mathbf{L}}}
\newcommand{\HH}{\operatorname{H}}

\newcommand{\shift}{\mathsf{\Sigma}}

\newcommand{\ideal}[1]{\mathfrak{#1}}

\newcommand{\fn}{\ideal{n}}

\newcommand{\xra}{\xrightarrow}

\renewcommand{\geq}{\geqslant}
\renewcommand{\leq}{\leqslant}

\newcommand{\Ext}[4][R]{\operatorname{Ext}_{#1}^{#2}(#3,#4)}

\newcommand{\Hom}{\operatorname{Hom}}

\def\Ext{\operatorname{Ext}}

\def\spec{\operatorname{Spec}}

\def\D{\operatorname{\mathsf{D}}}

\def\D{\mathcal{D}}
\def\C{\mathcal{C}}
\def\K{\mathcal{K}}

\def\Mon{\mathrm{Mon}}

\newcommand{\holim}{\operatorname{holim}}

\numberwithin{equation}{lem}

\begin{document}

\bibliographystyle{amsplain}

\title[Na\"{\i}ve liftings of DG modules]{Na\"{\i}ve liftings of DG modules}

\author{Saeed Nasseh}
\address{Department of Mathematical Sciences\\
Georgia Southern University\\
Statesboro, GA 30460, U.S.A.}
\email{snasseh@georgiasouthern.edu}

\author{Maiko Ono}
\address{Institute for the Advancement of Higher Education, Okayama University of Science, Ridaicho, Kitaku, Okayama 700-0005, Japan}
\email{ono@pub.ous.ac.jp}

\author{Yuji Yoshino}
\address{Graduate School of Natural Science and Technology, Okayama University, Okayama 700-8530, Japan}
\email{yoshino@math.okayama-u.ac.jp}

\thanks{Y. Yoshino was supported by JSPS Kakenhi Grant 19K03448.}



\begin{abstract}
Let $n$ be a positive integer, and let $A$ be a strongly commutative differential graded (DG) algebra over a commutative ring $R$. Assume that
\begin{enumerate}[\rm(a)]
\item
$B=A[X_1,\ldots,X_n]$ is a polynomial extension of $A$, where $X_1,\ldots,X_n$ are variables of positive degrees; or
\item
$A$ is a divided power DG $R$-algebra and $B=A \langle X_1,\ldots,X_n \rangle$ is a free extension of $A$ obtained by adjunction of variables $X_1,\ldots,X_n$ of positive degrees.
\end{enumerate}
In this paper, we study na\"{\i}ve liftability of DG modules along the natural injection $A\to B$ using the notions of diagonal ideals and homotopy limits. We prove that if $N$ is a bounded below semifree DG $B$-module such that $\Ext _B ^i (N, N)=0$  for all $i\geq 1$, then  $N$ is na\"{\i}vely liftable to $A$. This implies that $N$ is a direct summand of a DG $B$-module that is liftable to $A$. Also, the relation between na\"{\i}ve liftability of DG modules and the Auslander-Reiten Conjecture has been described.
\end{abstract}

\maketitle

\tableofcontents

\section{Introduction}\label{sec20200314a}

Throughout the paper, $R$ is a commutative ring.\vspace{5pt}

Let $I$ be an ideal of $R$, and assume in this paragraph that $R$ is $I$-adically complete and local. When $I$ is generated by an $R$-regular sequence, lifting property of finitely generated modules and of bounded below complexes of finitely generated free modules along the natural surjection $R\to R/I$ was studied by Auslander, Ding, and Solberg~\cite{auslander:lawlom} and Yoshino~\cite{yoshino}.
Nasseh and Sather-Wagstaff~\cite{nasseh:lql} generalized these results to the case where $I$ is not necessarily generated by an $R$-regular sequence. In this case, they considered the lifting property of differential graded (DG) modules along the natural map from $R$ to the Koszul complex on a set of generators of the ideal $I$.

Let $A\to B$ be a homomorphism of DG $R$-algebras. A right DG $B$-module $N$ is \emph{liftable} to $A$ if there is a right DG $A$-module $M$ such that $N \cong M\lotimes_A B$ (or $N\cong M\otimes_A B$, if $M$ and $N$ are semifree)
in the derived category $\D(B)$.
In their recent works, Nasseh and Yoshino~\cite{nassehyoshino} and Ono and Yoshino~\cite{OY} proved the following results on liftability of DG modules; see~\ref{para20201112a} and~\ref{para20201112c} for notation.

\begin{thm}[\cite{nassehyoshino, OY}]\label{thm20200605a}
Let $A$  be a DG $R$-algebra and $B=A\langle X\rangle$ be a simple free extension of $A$ obtained by adjunction of a variable $X$ of degree $|X|>0$ to kill a cycle in $A$. Assume that $N$ is a semifree DG $B$-module with $\Ext^{|X|+1}_B(N,N)=0$.
\begin{enumerate}[\rm(a)]
\item
If $|X|$ is odd, then $N\oplus N(-|X|)$ is liftable to $A$ (that is, $N$ is weakly liftable to $A$ in the sense of~\cite[Definition 5.1]{NOY}).
\item
If $|X|$ is even and $N$ is bounded below, then $N$ is liftable to $A$.
\end{enumerate}
\end{thm}

Na\"{\i}ve lifting property of DG modules along simple free extensions of DG algebras was introduced in~\cite{NOY} to obtain a new characterization of (weak) liftability of DG modules along such extensions; see~\cite[Theorem 6.8]{NOY}. However, our study of na\"ive lifting property of DG modules in this paper is mainly motivated by a conjecture of Auslander and Reiten as we explain in Section~\ref{sec20201126n}; see Theorem~\ref{thm20210108z}.
For the general definition of na\"{\i}ve liftability, let $A\to B$ be a homomorphism of DG $R$-algebras such that the underlying graded $A$-module $B$ is free.
Let $N$ be a semifree right DG $B$-module, and denote by $N |_A$ the DG $B$-module $N$ regarded as a right DG $A$-module via $A\to B$.
We say that $N$ is {\it na\"ively liftable}  to $A$ if
the DG $B$-module epimorphism
$\pi _N\colon N |_A \otimes _A B \to N$
defined by $\pi_N(x \otimes b)=xb$ splits; see~\ref{para20201113a} for more details.
The purpose of this paper is to prove the following result that deals with this version of liftability along finite free and polynomial extensions of DG algebras; see~\ref{para20201203a} and~\ref{para20201112b} for the definitions and notation.

\begin{mthm*}\label{thm20201114a}
Let $n$ be a positive integer. We consider the following two cases:
\begin{enumerate}[\rm(a)]
\item
$B=A[X_1,\ldots,X_n]$ is a polynomial extension of $A$, where $X_1,\ldots,X_n$ are variables of positive degrees; or
\item
$A$ is a divided power DG $R$-algebra and $B=A \langle X_1,\ldots,X_n \rangle$ is a free extension of $A$ obtained by adjunction of variables $X_1,\ldots,X_n$ of positive degrees.
\end{enumerate}
In either case, if $N$ is a bounded below semifree DG $B$-module with $\Ext _B ^i (N, N)=0$  for  all $i>0$, then  $N$ is na\"ively liftable to $A$. Moreover, $N$ is a direct sum of a DG $B$-module that is liftable to $A$.
\end{mthm*}

A unified method to prove parts (a) and (b) of Theorem~\ref{thm20200605a} is introduced in~\cite{NOY} using the notion of $j$-operators. However, as is noted in~\cite[3.10]{NOY}, this notion cannot be generalized (in a way that useful properties of $j$-operators are preserved) to the case where we have more than one variable. Our approach in this paper in order to prove Main Theorem is as follows. In Section~\ref{sec20201126a}, we define the notions of diagonal ideals and DG smoothness, which is a generalization of the notion of smooth algebras in commutative ring theory. Then using the notion of homotopy limits, discussed in Section~\ref{sec20201126b}, we prove the following result in Section~\ref{sec20201126c}.

\begin{thm}\label{thm2021017a}
Let $A\to B$ be a DG smooth homomorphism. If $N$ is a bounded below semifree DG $B$-module with $\Ext _B ^i (N, N)=0$  for  all $i\geq 1$, then  $N$ is na\"ively liftable to $A$. Moreover, $N$ is a direct sum of a DG $B$-module that is liftable to $A$.
\end{thm}

The proof of Main Theorem then follows after we show that under the assumptions of Main Theorem, $A\to B$ is DG smooth. This takes up the entire Section~\ref{sec20210107a}.

\section{Terminology and preliminaries}\label{sec20200314b}

We assume that the reader is fairly familiar with complexes, DG algebras, DG modules, and their properties. Some of the references on these subjects are~\cite{avramov:ifr,avramov:dgha, felix:rht, GL}. In this section, we specify the terminology and include some preliminaries that will be used in the subsequent sections.

\begin{para}\label{para20200329a}
Throughout the paper, $A$ is a \emph{strongly commutative differential graded $R$-algebra} (\emph{DG $R$-algebra}, for short), that is,
\begin{enumerate}[\rm(a)]
\item
$A  = \bigoplus  _{n \geq 0} A _n$ is a non-negatively \emph{graded commutative} $R$-algebra\footnote{Some authors use the cohomological notaion for DG algebras. In such a case, $A$ is described as  $A  = \bigoplus  _{n \leq 0} A ^n$, where $A^{n} = A_{-n}$ and $A$ is called non-positively graded.}, i.e., for all homogeneous elements $a, b \in A$ we have $ab = (-1)^{|a| |b|}ba$, and  $a^2 =0$  if the degree of $a$ (denoted $|a|$) is odd;
\item
$A$ is an $R$-complex with a differential $d^A$ (that is, a graded $R$-linear map $A\to A$ of degree $-1$ with $(d^A)^2=0$); such that
\item
$d^A$ satisfies the \emph{Leibniz rule}: for all homogeneous elements $a,b\in A$ the equality $d^A(ab) = d^A(a) b + (-1)^{|a|}ad^A(b)$ holds.
\end{enumerate}
A \emph{homomorphism} $f\colon A\to B$ of DG $R$-algebras is a graded $R$-algebra homomorphism of degree $0$ which is also a chain map, that is, $d^Bf=fd^A$.
\end{para}

\begin{para}\label{para20201203a}
An $R$-algebra $U$ is a \emph{divided power algebra} if a sequence of elements $u^{(i)}\in U$ with $i\in \mathbb{N}\cup \{0\}$ is correspondent to every element $u\in U$ with $|u|$ positive and even such that the following conditions are satisfied:
\begin{enumerate}[\rm(1)]
\item
$u^{(0)}=1$, $u^{(1)}=u$, and $|u^{(i)}|=i|u|$ for all $i$;
\item
$u^{(i)}u^{(j)}=\binom{i+j}{i}u^{(i+j)}$ for all $i,j$;
\item
$(u+v)^{(i)}=\sum_{j}u^{(j)}v^{(i-j)}$ for all $i$;
\item
for all $i\geq 2$ we have
$$
(vw)^{(i)}=
\begin{cases}
0& |v|\ \text{and}\ |w|\ \text{are odd}\\
v^iw^{(i)}& |v|\ \text{is even and}\ |w|\ \text{is even and positive}
\end{cases}
$$
\item
For all $i\geq 1$ and $j\geq 0$ we have
$$\left(u^{(i)}\right)^{(j)}=\frac{(ij)!}{j!(i!)^j}u^{(ij)}.$$
\end{enumerate}
A \emph{divided power DG $R$-algebra} is a DG $R$-algebra whose underlying graded $R$-algebra is a divided power algebra.
\end{para}

\begin{para}\label{para20201126d}
If $R$ contains the field of rational numbers and $U$ is a graded $R$-algebra, then $U$ has a structure of a divided power $R$-algebra by defining $u^{(m)}=(1/m!)u^m$ for all $u\in U$ and integers $m\geq 0$; see~\cite[Lemma 1.7.2]{GL}. Also, $R$ considered as a graded $R$-algebra concentrated in degree $0$ is a divided power $R$-algebra.
\end{para}

\begin{para}\label{para20201112a}
Let $t\in A$ be a cycle, and let $A \langle X \rangle$ with the differential $d$ denote the \emph{simple free extension of $A$} obtained by adjunction of a variable $X$ of degree $|t|+1$ such that $dX = t$. The DG $R$-algebra $A \langle X \rangle$ can be described as $A \langle X \rangle = \bigoplus_{m\geq 0} X^{(m)}A$ with the conventions $X^{(0)}=1$ and $X^{(1)}=X$, where $\{X^{(m)}\mid m\geq 0\}$ is a free basis of $A\langle X\rangle$ such that:
\begin{enumerate}[\rm(a)]
\item
If $|X|$ is odd, then $X^{(m)}=0$ for all $m\geq 2$, and for all $a + Xb\in A \langle X \rangle$ we have
$$
d(a + Xb)=d^Aa + tb - X d^Ab.
$$

\item
If $|X|$ is even, then $A \langle X \rangle$ is a divided power DG $R$-algebra with the algebra structure given by $X^{(m)}X^{(\ell)} =\binom{m+\ell}{m} X^{(m+\ell)}$ and the differential structure defined by $dX^{(m)}=X^{(m-1)}t$ for all $m\geq 1$.
\end{enumerate}

Also, let $A [X]$ denote the \emph{simple polynomial extension of $A$} with $X$ described as above, that is, $A [X] = \bigoplus_{m\geq 0} X^{m}A$ with $d^{A [X]}(X^m)=mX^{m-1}t$ for positive integers $m$. Note that here $X^m$ is just the ordinary power on $X$.

If $R$ contains the field of rational numbers, then $A\langle X\rangle=A[X]$.
\end{para}

\begin{para}\label{para20201112b}
Let $n$ be a positive integer, and let $A \langle X_1,\ldots,X_n\rangle$ (which is also denoted by $A \langle X_i \mid 1\leq i\leq n\rangle$) be a \emph{finite free extension of the DG $R$-algebra $A$} obtained by adjunction of $n$ variables. In fact, setting $A^{(0)} =A$ and $A^{(i)}=A^{(i-1)}\langle X_i \rangle$  for all $1 \leq i \leq n$ such that $d^{A^{(i)}}X_i$ is a cycle in $A^{(i-1)}$, we have $A \langle X_1,\ldots,X_n\rangle=A^{(n)}$. We also assume that  $0 < |X_1| \leq  \cdots \leq |X_n|$. Note that there is
a sequence of DG $R$-algebras $A= A^{(0)} \subset A ^{(1)} \subset \cdots \subset A^{(n)}=A \langle X_1,\ldots,X_n\rangle$.

In a similar way, one can define the \emph{finite polynomial extension of the DG $R$-algebra $A$}, which is denoted by $A [X_1,\ldots,X_n]$.
\end{para}

\begin{para}\label{para20201205a}
Our discussion in~\ref{para20201112b} can be extended to the case of adjunction of infinitely countably many variables to the DG $R$-algebra $A$.
Let $\{ X_i \mid i \in \mathbb{N} \}$ be a set of variables.
Attaching a degree to each variable such that $0 < |X_1| \leq  |X_2| \leq  \cdots$, similar to~\ref{para20201112b},
we construct a sequence
$A= A^{(0)} \subset A ^{(1)} \subset  A^{(2)}  \subset  \cdots$
of DG $R$-algebras.
We define an \emph{infinite free extension of the DG $R$-algebra $A$} obtained by adjunction of the variables $X_1, X_2, \ldots$ to be
$A \langle X_i \mid  i\in \mathbb{N} \rangle = \bigcup _{n \in \mathbb{N}} A^{(n)}$. It is sometimes convenient for us to use the notation $A \langle X_1,\ldots,X_n\rangle$ with $n=\infty$ instead of $A \langle X_i \mid  i\in \mathbb{N} \rangle $.

For the infinite extension $A \langle X_i \mid  i\in \mathbb{N} \rangle $ of the DG $R$-algebra $A$, we always assume the {\it degree-wise finiteness condition}, that is, for all $n\in \mathbb{N}$, we assume that the set $\{ i \mid  |X_i| = n \}$ is finite. As an example of this situation, let  $R \to S$ be a surjective ring homomorphism of commutative noetherian rings.
Then the Tate resolution of $S$ over  $R$  is an extension of the DG $R$-algebra $R$ (with infinitely countably many variables, in general) which satisfies the degree-wise finiteness condition; see~\cite{Tate}.

In a similar way, one can define the \emph{infinite polynomial extension of the DG $R$-algebra $A$}, which is denoted by $A [X_i\mid i\in \mathbb{N}]$ or $A [X_1,\ldots,X_n]$ with $n=\infty$.
\end{para}

\begin{para}\label{para20201124a}
For $n\leq \infty$, let $\Gamma=\bigcup_{i=1}^{n}\{X_i^{(m)}\mid m\geq 0\}$ with the conventions from~\ref{para20201112a} that if $|X_i|$ is odd, then $X_i^{(0)}=1$, $X_i^{(1)}=X_i$, and $X_i^{(m)}=0$ for all $m\geq 2$.

If $n<\infty$, then the set $\{ X_1 ^{(m_1)}X_2^{(m_2)}\cdots X_n^{(m_n)}\mid  X_i^{(m_i)}\in \Gamma\ (1 \leq i \leq n)\}$ is a basis for the underlying graded free $A$-module $A \langle X_1,\ldots,X_n\rangle$.

If $n=\infty$, then the set $\{ X_{i_1} ^{(m_{i_1})}X_{i_2}^{(m_{i_2})}\cdots X_{i_t}^{(m_{i_t})}\mid  X_{i_j}^{(m_{i_j})}\in \Gamma\ (i_j\in \mathbb{N}, t<\infty)\}$ is a basis for the underlying graded free $A$-module $A \langle X_i \mid  i\in \mathbb{N} \rangle$.

The cases $A [X_1,\ldots,X_n]$ and $A [X_i\mid i\in \mathbb{N}]$ can be treated similarly by using ordinary powers $X_i^m$ instead of divided powers $X_i^{(m)}$.
\end{para}

\begin{para}\label{para20201112c}
A right \emph{DG $A$-module} $(M, \partial^M)$ (or simply $M$) is a graded right $A$-module $M=\bigoplus_{i\in \mathbb{Z}}M_i$ that is also an $R$-complex with the differential $\partial^M$ satisfying the Leibniz rule, that is, the equality
$\partial^M(ma) = \partial^M(m)\ a + (-1)^{|m|} m\ d^A(a)$ holds for all homogeneous elements $a\in A$ and $m\in M$.

All DG modules considered in this paper are right DG modules, unless otherwise stated. Since $A$ is graded commutative, a DG $A$-module $M$ is also a left DG $A$-module with the left $A$-action
defined by $am = (-1)^{|m||a|} ma$  for  $a\in A$ and $m \in M$.

A \emph{DG submodule} of a DG $A$-module $M$ is a subcomplex that is a DG $A$-module under the operations induced by $M$, and a \emph{DG ideal} of $A$ is a DG submodule of $A$.

For a DG $A$-module $M$, let $\inf (M) = \inf \{ i \in \mathbb{Z}\mid  M_i \not=0 \}$. We say that $M$ is \emph{bounded below} if $\inf (M) > -\infty$, that is, if $M_i=0$ for all $i\ll 0$.
Note that $\inf (L) \geq \inf (M)$ if $L$ is a DG $A$-submodule of $M$. For an integer $i$, the \emph{$i$-th shift} of $M$, denoted $\shift^i M$ or $M(-i)$, is defined by $\left(\shift^i M\right)_j = M_{j-i}$ with $\partial_j^{\shift^i M}=(-1)^i\partial_{j-i}^M$.
\end{para}

\begin{para}\label{para20201114a}
Let $A^{o}$ denote the \emph{opposite DG $R$-algebra} which is equal to $A$ as a set, but to distinguish elements in $A^o$ and $A$ we write $a^o\in A^o$ if $a \in A$.
The product of elements in $A^o$ and the differential $d^{A^o}$ are given by the formulas
$a^ob^o = (-1)^{|a||b|}(ba)^o=(ab)^o$ and $d^{A^o} (a^o) = d^A(a)^o$, for all homogeneous elements $a, b \in A$.
Since $A$ is a graded commutative DG $R$-algebra, the identity map  $A \to A^o$ that corresponds $a \in A$ to $a^o\in A^o$ is a DG $R$-algebra isomorphism.
From this point of view, there is no need to distinguish between $A$ and $A^o$.
However, we will continue using the notation $A^o$ to make it clear how we use the graded commutativity of $A$.

Note that every right (resp. left) DG $A$-module $M$ is a left (resp. right) DG $A^o$-module with $a^om=(-1)^{|a^o||m|}ma$ (resp. $ma^o=(-1)^{|a^o||m|}am$) for all homogeneous elements $a\in A$ and $m\in M$.
\end{para}

\begin{para}\label{para20201114e}
Let $A\to B$ be a homomorphism of DG $R$-algebras such that $B$ is projective as an underlying graded $A$-module. Let $B^e$ denote the \emph{enveloping DG $R$-algebra} $B^o \otimes_A B$ of $B$ over $A$.
The algebra structure on $B^e$ is given by
$$
 (b_1^o \otimes b_2)( {b'}_1^o \otimes {b'}_2)
 =  (-1)^{|{b'}_1| |b_2|} b_1^o {b'}_1^o \otimes b_2{b'}_2
=  (-1)^{|{b'}_1| |b_2|+|{b'}_1| |b_1|}({b'}_1 b_1)^o\otimes b_2{b'}_2
$$
for all homogeneous elements $b_1, b_2, b'_1, b'_2\in B$, while the graded structure is given by $(B^e)_i  = \sum_{j} (B^o)_j\otimes_A B_{i-j}$ and the differential $d^{B^e}$ is defined by
$d^{B^e} (  b_1^o \otimes b_2 ) = d^{B^o} (  b_1^o) \otimes b_2 + (-1)^{|b_1|} b_1^o \otimes d^{B}(b_2 )$.

Note that $B$ and $B^o$ are regarded as subrings of $B^e$.
Moreover, the map $B^o \to B^e$ defined by $b^o\mapsto b ^o \otimes 1$ is an injective DG $R$-algebra homomorphism, via which we can consider $B^o$  as a DG $R$-subalgebra of $B^e$.
Since $B$ is graded commutative, $B \cong B^o$ and hence, $B$ is a DG $R$-subalgebra of $B^e$ as well.

Note also that DG $B^e$-modules are precisely DG $(B, B)$-bimodules. In fact, for a DG $B^e$-module $N$, the right action of an element of $B^e$ on $N$ yields the two-sided module structure $n ( b_1 ^o \otimes b_2) = (-1)^{|b_1||n|}b_1 n b_2$ for all homogeneous elements $n \in N$ and $b_1, b_2 \in B$.
Hence, the differential $\partial ^N$ satisfies the Leibniz rule on both sides:
$\partial ^N ( b_1nb_2 ) = d^B(b_1)nb_2 + (-1)^{|b_1|} b_1 \partial^N(n)b_2 + (-1)^{|b_1|+|n|} b_1n d^B(b_2)$
for all homogeneous elements $n \in N$ and $b_1, b_2 \in B$.
\end{para}

\begin{para}\label{para20201206d}
Consider the notation from~\ref{para20201124a} and~\ref{para20201114e}. Let
\begin{align*}
\Mon(\Gamma)\!\!&=\!\!
\begin{cases}
\!\!\{ (1^o \otimes X_1 ^{(m_1)})\cdots (1^o \otimes X_n^{(m_n)})\mid  X_i^{(m_i)}\in \Gamma\ (1 \leq i \leq n) \}&\!\!\!\text{if}\ n<\infty\\
\!\!\{ (1^o \otimes X_{i_1} ^{(m_{i_1})})\cdots (1^o \otimes X_{i_t}^{(m_{i_t})})\mid  X_{i_j}^{(m_{i_j})}\in \Gamma\ (i_j\in \mathbb{N}, t<\infty) \}&\!\!\!\text{if}\ n=\infty
\end{cases}\\
&=
\begin{cases}
\{ 1^o \otimes (X_1 ^{(m_1)}\cdots X_n^{(m_n)})\mid  X_i^{(m_i)}\in \Gamma\ (1 \leq i \leq n) \}&\text{if}\ n<\infty\\
\{ 1^o \otimes (X_{i_1} ^{(m_{i_1})} \cdots X_{i_t}^{(m_{i_t})})\mid  X_{i_j}^{(m_{i_j})}\in \Gamma\ (i_j\in \mathbb{N}, t<\infty) \}&\text{if}\ n=\infty
\end{cases}
\end{align*}
Then the underlying graded $A \langle X_1,\ldots,X_n\rangle^o$-module $A \langle X_1,\ldots,X_n\rangle^e$ with $n\leq \infty$ is free with the basis $\Mon(\Gamma)$.\footnote{``$\Mon$'' is chosen for ``monomial.''}

Once again, the case of $A [X_1,\ldots,X_n]$ with $n\leq \infty$ can be treated similarly by using $X_i^m$ instead of $X_i^{(m)}$.
\end{para}

\begin{para}\label{para20201124b}
A \emph{semifree basis} (or \emph{semi-basis}) of a DG $A$-module $M$ is a well-ordered subset $F\subseteq M$ that is a basis for the underlying graded $A$-module $M$ and satisfies $\partial^M(f)\in \sum_{e<f}eA$ for every element $f\in F$. A DG $A$-module $M$ is \emph{semifree}\footnote{Keller~\cite{keller:ddgc} calls these ``DG modules that have property (P).''} if it has a semifree basis. Equivalently, the DG $A$-module $M$ is semifree if there exists an increasing filtration $$0=F_{-1}\subseteq F_0\subseteq F_1\subseteq \cdots\subseteq M$$ of DG $A$-submodules of $M$ such that $M=\bigcup_{i\geq 0}F_i$ and each DG $A$-module $F_i/F_{i-1}$ is a direct sum of copies of $A(n)$ with $n\in \mathbb{Z}$; see~\cite{AH},~\cite[A.2]{AINSW}, or~\cite{felix:rht}.
\end{para}

\begin{para}\label{para20201124c}
Let $\C(A)$ denote the abelian category of DG $A$-modules and DG $A$-module homomorphisms.
Also, let $\K(A)$ be the \emph{homotopy category} of DG $A$-modules. Recall that objects of $\K(A)$ are DG $A$-modules and
morphisms are the set of homotopy equivalence classes of DG $A$-module homomorphisms $\Hom _{\K(A)} (M, L) = \Hom _{\C(A)} (M, L)/ \sim$,
where $f \sim g$ for $f,g\in \Hom _{\C(A)} (M, L)$ if and only if there is a graded $A$-module homomorphism $h\colon M \to L (-1)$ of underlying graded $A$-modules such that $f - g = \partial ^L h + h \partial ^M$. It is known that $\K(A)$ is triangulated category.
In fact, there is a triangle $M \to L \to Z \to \shift M$ in $\K(A)$
if and only if there is a short exact sequence
$0 \to M \to L \oplus L' \to Z \to 0$
in $\C(A)$ in which $L'$ is splitting exact, i.e., $\id_{L'}\sim 0$. The \emph{derived category} $\D(A)$ is obtained from $\C(A)$ by formally inverting the quasi-isomorphisms (denoted $\simeq$); see, for instance,~\cite{keller:ddgc} for details.

For each integer $i$ and DG $A$-modules $M,L$ with $M$ being semifree, $\Ext^i_A(M,L)$ is defined to be $\HH_{-i}\left(\Hom_A(M,L)\right)$. Note that $\Ext^i_A(M,L)=   \Hom _{\K(A)}(M, L(-i))$.
\end{para}

\section{Diagonal ideals and DG smoothness}\label{sec20201126a}

In this section, we introduce the notion of diagonal ideals which play an essential role in the proofs of Theorem~\ref{thm2021017a} and Main Theorem.


\begin{para}\label{para20201114d}
Let $\varphi\colon A\to B$ be a homomorphism of DG $R$-algebras such that $B$ is projective as an underlying graded $A$-module. Let $\pi _B\colon B^e \to B$ denote the map defined by $\pi _B(b_1 ^o \otimes b_2) = b_1b_2$.
For all homogeneous elements $b_1, b_2, b'_1, b'_2\in B$ we have
\begin{eqnarray*}
\pi_B((b_1 ^o \otimes b_2)({b'}_1 ^o \otimes {b'}_2))
&=&(-1) ^{|{b'}_1||b_2|+|{b'}_1||b_1|} \pi_B(({b'}_1  {b}_1) ^o \otimes b_2{b'}_2)  \\
&=& (-1) ^{|{b'}_1||b_2|+|{b'}_1||b_1|}({b'}_1 {b}_1) (b_2{b'}_2) \\
&=& ({b}_1 {b}_2) ({b'}_1{b'}_2) \\
&=& \pi_B(b_1 ^o \otimes b_2) \pi _B({b'}_1 ^o \otimes {b'}_2).
\end{eqnarray*}
Hence, $\pi _B$ is an algebra homomorphism.
Also, it is straightforward to check that $\pi _B$ is a chain map.
Therefore,  $\pi_B$  is a homomorphism of DG $R$-algebras.
\end{para}

\begin{defn}\label{defn20201206a}
In the setting of~\ref{para20201114d}, kernel of $\pi_B$ is denoted by $J=J_{B/A}$ and is called the \emph{diagonal ideal} of $\varphi$.\footnote{The definition of diagonal ideals originates in  scheme theory. In fact, if $A \to B$ is a homomorphism of commutative rings, then the kernel of the natural mapping $B \otimes _AB \to B$ is the defining ideal of the diagonal set in the Cartesian product $\spec B \times _{\spec A} \spec B$.
}
\end{defn}

\begin{para}\label{para20201206a}
In~\ref{para20201114d}, since $\pi_B$ is a homomorphism of DG $R$-algebras, $J$ is a DG ideal of $B^e$.
The isomorphism $B^e /J \cong B$ of DG $R$-algebras is also an isomorphism of DG $B^e$-modules. Hence, there is an exact sequence of DG $B^e$-modules:
\begin{equation}\label{eq20201114a}
0 \to J \to  B^e \xra{\pi_B} B  \to 0.
\end{equation}
\end{para}

Next, we define our notion of smoothness for DG algebras.

\begin{defn}\label{defn20210105a}
Let $\varphi\colon A\to B$ be a homomorphism of DG $R$-algebras.
We say that $B$ is \emph{DG quasi-smooth over $A$} (or simply \emph{$\varphi$ is DG quasi-smooth})
if the following conditions are satisfied:
\begin{enumerate}[\rm(i)]
\item
$B$ is free as an underlying graded $A$-module.
\item
The diagonal ideal $J$ has a filtration consisting of DG $B^e$-submodules\footnote{$J^{[\ell]}$ is just a notation for the $\ell$-th DG $B^e$-submodule of $J$ in the sequence. It is not an $\ell$-th power of any kind.}
$$
J =J^{[1]} \supset J^{[2]} \supset J^{[3]} \supset  \cdots \supset J^{[\ell]} \supset J^{[\ell+1]} \supset \cdots
$$
such  that  $J J^{[\ell]} + J^{[\ell]} J \subseteq J^{[\ell+1]}$ for all $\ell \geq 1$, and each element of $J^{[\ell]}$ has degree $\geq \ell$, that is, $\inf (J^{[\ell]})\geq \ell$. This implies that $\bigcap _\ell J^{[\ell]} =(0)$.
\item
For every $\ell \geq 1$, the DG $B$-module $J^{[\ell]}/J^{[\ell+1]}$ is semifree.
\end{enumerate}

We say that $B$ is \emph{DG smooth over $A$} (or simply \emph{$\varphi$ is DG smooth}) if it is DG quasi-smooth over $A$ and for all positive integers $\ell$, the semifree DG $B$-module $J^{[\ell]}/J^{[\ell+1]}$ has a finite semifree basis.\footnote{In case that $A \to B$ is a homomorphism of commutative rings, $B$ is projective over $A$, and  $J/J^2$ is projective over $B$, then $B$ is smooth over $A$ in the sense of scheme theory. In this case, $J/J^2 \cong \Omega _{B/A}$ is the module of K\"{a}hler differentials.}
\end{defn}

\begin{para}
There exist other definitions of smoothness for DG algebras. For instance, a definition given by Kontsevich is found in~\cite{Kont} (alternatively,  in~\cite[Section 18]{Yekutieli}). Also, another version of smoothness for DG algebras is introduced by To\"{e}n and Vezzosi in~\cite{TV} which Shaul~\cite{Shaul} proves is equivalent to Kontsevich's definition. However, our above version of smoothness is new and quite different from any existing definition of smoothness for DG algebras.
\end{para}

\begin{para}\label{para20210130a}
Let $\varphi\colon A\to B$ be a homomorphism of DG $R$-algebras. If $B$ is DG smooth over $A$, then for any integer $\ell \geq 1$,
there is a finite filtration
$$
J = L_0 \supset L_1 \supset L_2 \supset \cdots \supset L _s \supset L_{s+1}= J ^{[\ell]}
$$
of $J$ by its DG $B^e$-submodules, where for each $0 \leq i \leq s$ we have $L_i /L_{i+1} \cong B(-a_i)$ as DG $B^e$-modules, for some positive integer $a_i$.
\end{para}

\begin{para}\label{para20210106a}
We will show in Section~\ref{sec20210107a} that free extensions of divided power DG $R$-algebras and polynomial extensions of DG $R$-algebras are DG quasi-smooth. If these extensions are finite, then we have the DG smooth property; see Corollary~\ref{cor20210105a}.

There are several examples of DG smooth extensions besides free or polynomial extensions.
For instance, as one of the most trivial examples,
let $B = A \langle X  \rangle /(X^2)$, where $|X|$ is even and $d^BX =0$.
If $A$ contains a field of characteristic $2$, then $B$ is DG smooth over $A$ by setting $J^{[2]} =(0)$.
\end{para}

\begin{para}\label{para20201126a}
Let $\varphi\colon A\to B$ be a DG quasi-smooth homomorphism, and use the notation of Definition~\ref{defn20210105a}. Let $N$ be a semifree DG $B$-module. For every positive integer $\ell$, consider the DG $B^e$-module $J/J^{[\ell]}$ as a DG $(B,B)$-bimodule. The tensor product $N \otimes _B  J/J^{[\ell]}$ uses the left DG $B$-module structure of  $J/J^{[\ell]}$ and in this situation, $N \otimes _B  J/J^{[\ell]}$ is a right DG $B$-module by the right $B$-action on $J/J^{[\ell]}$.
\end{para}

The following lemma is useful in the next section.

\begin{lem}\label{lemma for HLtheorem}
Let $\varphi\colon A\to B$ be a DG quasi-smooth homomorphism, and use the notation of Definition~\ref{defn20210105a}. Suppose that $N$ is a semifree DG $B$-module such that $\Ext_B ^i (N,N \otimes _B  J/J^{[\ell]})=0$ for all $i \geq 0$ and some $\ell \geq 1$. Then the natural inclusion $J ^{[\ell]} \hookrightarrow J$ induces an isomorphism  $\Ext_B ^i (N,N \otimes _B  J^{[\ell]}) \cong \Ext_B ^i (N,N \otimes _B  J)$ for all $i\geq 1$.
\end{lem}

\begin{proof}
Applying $N \otimes _B - $ to the short exact sequence
$0 \to J^{[\ell]} \to J \to J/J^{[\ell]} \to 0$
of DG $B^e$-modules, we get an exact sequence
\begin{equation}\label{eq20201121a}
0 \to N \otimes _B J^{[\ell]} \to N \otimes _B J \to N \otimes _B J/J^{[\ell]} \to 0
\end{equation}
of DG $B$-modules, where the injectivity on the left comes from the fact that $N$ is free as an underlying graded $B$-module.
The assertion now follows from the long exact sequence of Ext obtained from applying $\Hom _B ( N , - )$ to~\eqref{eq20201121a}.
\end{proof}

The following result is used in the proof of Theorem~\ref{thm2021017a}.

\begin{thm}\label{Extzero}
Let $\varphi\colon A\to B$ be a DG smooth homomorphism, and use the notation of Definition~\ref{defn20210105a}. Let $N$ be a semifree DG $B$-module with $\Ext_B ^i (N,N)=0$ for all $i \geq 1$. Then for all $i \geq 0$ and all $\ell \geq 1$ we have $\Ext_B ^i (N,N \otimes _B  J^{[\ell]}/J^{[\ell +1]})=0=\Ext_B ^i (N,N \otimes _B  J/J^{[\ell]})$.
\end{thm}

\begin{proof}
We treat both of the equalities at the same time. Let $L$ denote $J^{[\ell]}/J^{[\ell +1]}$ or $J/J^{[\ell]}$ with $\ell\geq 1$.
By definition of DG smoothness, there is a finite filtration
$$
L = L_0 \supset L_1 \supset L_2 \supset \cdots \supset L _s \supset L_{s+1}=(0)
$$
of $L$ by its DG $B^e$-submodules, where for each $0 \leq i \leq s$ we have $L_i /L_{i+1} \cong B(-a_i)$ as DG $B^e$-modules, for some positive integer $a_i$.

We now prove by induction on $s$ that  $\Ext_B ^i (N,N \otimes _B  L)=0$ for all $i \geq 0$.
For the base case where $s=0$, we have $L \cong B(-a_0)$. Hence, $N\otimes _B L \cong N(-a_0)$.
Therefore, $\Ext_B ^i (N,N \otimes _B L) \cong \Ext_B ^i (N,N (-a_0)) = \Ext _B ^{i+ a_0} (N, N) = 0$ for all $i \geq 0$.

Assume now that $s \geq 1$. Since $N$ is a semifree DG $B$-module, tensoring the short exact sequence $0 \to L_{1} \to L \to B (-a_0) \to 0$ of DG $B^e$-modules (hence, DG $B$-modules) by $N$, we get a short exact sequence
\begin{equation}\label{eq20201116a}
0 \to N\otimes_BL_{1} \to N\otimes_BL \to N\otimes_BB (-a_0) \to 0
\end{equation}
of DG $B$-modules (hence, DG $B^e$-modules via $\pi_B$).
By inductive hypothesis we have $\Ext_B ^i (N,N \otimes _B L_{1})=0$ for all $i\geq 0$. Also, $\Ext_B ^i (N,N \otimes _B B(-a_0)) = 0$ for all $i \geq 0$ by the base case $s=0$.
It follows from the long exact sequence of cohomology modules obtained from~\eqref{eq20201116a} that $\Ext_B ^i (N,N \otimes _B  L) =0$ for all $i \geq 0$.
\end{proof}

The next result is also crucial in the proof of Theorem~\ref{thm2021017a}.

\begin{thm}\label{HLtheorem}
Let $\varphi\colon A\to B$ be a DG quasi-smooth homomorphism, and use the notation of Definition~\ref{defn20210105a}. Let $N$ be a bounded below semifree DG $B$-module with $\Ext_B ^i (N,N \otimes _B  J/J^{[\ell]})=0$ for all $i \geq 0$ and all $\ell\geq 1$. Then $\Ext_B ^i (N,N \otimes _B  J)=0$ for all $i\geq 1$.
\end{thm}

The proof of this result needs the machinery of homotopy limits, which we discuss in the next section. We give the proof of this theorem in~\ref{para20201122d} below.

\section{Homotopy limits and proof of Theorem~\ref{HLtheorem}}\label{sec20201126b}

The entire section is devoted to the proof of Theorem~\ref{HLtheorem}.
The notion of homotopy limits, which we define in~\ref{para20201122b}, plays an essential role in the proof of the following result which is a key to the proof of Theorem~\ref{HLtheorem}.

\begin{thm}\label{limit}
Let $M$ and $N$ be DG $A$-modules with $M$ bounded below and $N$ semifree. Assume that there is a descending sequence
$$
M = M^0 \supseteq  M^1\supseteq  M^2 \supseteq \cdots \supseteq M ^{\ell} \supseteq M ^{\ell +1} \supseteq \cdots
$$
of DG $A$-submodules of $M$ that satisfies the following conditions:
\begin{enumerate}[\rm(1)]
\item  $\lim_{\ell\to \infty}\inf (M^{\ell})=\infty$; and
\item  there is an integer $k$ such that  the natural maps $M ^{\ell} \hookrightarrow M$ induce isomorphisms
$\Ext_A ^k ( N, M^{\ell}) \cong \Ext_A ^k ( N, M)$  for all $\ell \geq 1$.
\end{enumerate}
Then $\Ext_A ^k ( N, M) = 0$.
\end{thm}


By~\ref{para20201124c}, Theorem~\ref{limit} can be restated as follows. We prove this result in~\ref{para20201122c}.

\begin{thm}\label{limit2}
Let $M$ and $N$ be DG $A$-modules with $M$ bounded below and $N$ semifree. Assume that there is a descending sequence
$$M = M^0 \supseteq  M^1\supseteq  M^2 \supseteq \cdots \supseteq M ^{\ell} \supseteq M ^{\ell +1} \supseteq \cdots$$
of DG $A$-submodules of $M$ that satisfies the following conditions:
\begin{enumerate}[\rm(1)]
\item  $\lim_{\ell\to \infty}\inf (M^{\ell})=\infty$; and
\item  for all positive integers $\ell$, the natural maps $M ^{\ell} \hookrightarrow M$ induce isomorphisms
$\Hom _{\K(A)}( N, M^{\ell}) \cong \Hom _{\K(A)}( N, M)$.
\end{enumerate}
Then $\Hom _{\K(A)} ( N, M) = 0$.
\end{thm}

\begin{para}\label{para20201122v}
For a family $\{ M ^{\ell} \mid \ell \in \mathbb{N} \}$  of countably many DG $A$-modules,
the \emph{product} (or the \emph{direct product}) $P = \prod _{\ell \in \mathbb{N}} M^{\ell}$ in $\C(A)$ is constructed as follows:
the DG $A$-module $P$ has a $\mathbb{Z}$-graded structure
$P_i = \prod _{\ell \in \mathbb{N}} \left(M ^{\ell}\right)_i$ for all $i\in \mathbb{Z}$ with the differential that is given by the formula
$$\partial_i ^P \left( (m^{\ell})_{\ell \in \mathbb{N}} \right) = \left( \partial_i ^{M^{\ell}}(m^{\ell})  \right) _{\ell \in \mathbb{N}}$$
for all $(m^{\ell})_{\ell \in \mathbb{N}} \in P_i$.
By definition, we have
$$\Hom _{\C(A)} ( - , P) \cong \prod _{\ell \in \mathbb{N}} \Hom _{\C(A)} ( - , M^{\ell})$$
as functors on $\C(A)$. 
It can be seen that $P$ is also a product in $\K(A)$. Hence,
\begin{equation}\label{product}
\Hom _{\K(A)} ( - , P) \cong \prod _{\ell \in \mathbb{N}} \Hom _{\K(A)} ( - , M^{\ell})
\end{equation}
as functors on $\K(A)$.
\end{para}

Next, we define the notion of homotopy limits; references on this include~\cite{avramov:dgha, keller:ddgc}.

\begin{para}\label{para20201122b}
Assume that  $\{M ^{\ell} \mid \ell \in \mathbb{N}\}$ is a family of DG $A$-modules such that
$$ M^1\supseteq  M^2 \supseteq \cdots \supseteq M ^{\ell} \supseteq M ^{\ell +1} \supseteq \cdots.$$
Then, the \emph{homotopy limit}  $L= \holim  M^{\ell}$  is defined by the triangle in $\K(A)$
\begin{equation}\label{holimit}
L \to  P \xra{\varphi} P \to \shift L
\end{equation}
where $P = \prod _{\ell \in \mathbb{N}} M^{\ell}$ is the product introduced in~\ref{para20201122v} and $\varphi$ is defined by
$$
\varphi \left( (m^{\ell} \right) _{\ell \in \mathbb{N}} ) = \left( m^{\ell} - m^{\ell +1} \right)_{\ell \in \mathbb{N}}.
$$
Note that~\eqref{holimit} is a triangle in the derived category $\D(A)$ as well.

Let  $N$ be a semifree DG $A$-module.
Since $N$ is $\K(A)$-projective, we note that  $\Hom _{\D(A)} (N, - ) = \Hom _{\K(A)} (N, - )$ on the object set of $\C(A)$.
Applying the functor $\Hom _{\K(A)} (N, - )$ to the triangle~\eqref{holimit}, by~\ref{para20201122v} we have a triangle
{\small
$$
\Hom _{\K(A)} (N, L) \to  \prod _{\ell \in \mathbb{N}} \Hom _{\K(A)} (N, M^{\ell} ) \to \prod _{\ell \in \mathbb{N}} \Hom _{\K(A)} (N, M^{\ell} ) \to \shift \Hom _{\K(A)} (N, L)
$$}
in $\D(R)$.
Therefore, $\Hom _{\K(A)} (N , \holim M^{\ell}) \cong \holim \Hom _{\K(A)} (N ,  M^{\ell})$ in $\D(R)$.
\end{para}

\begin{lem}\label{lem20201122a}
Under the assumptions of Theorem~\ref{limit2} we have $\HH (\holim M^{\ell}) =0$, that is, $\holim M^{\ell}$ is zero in the derived category $\D(A)$.
\end{lem}

\begin{proof}
Let $L = \holim M^{\ell}$.
The triangle~\eqref{holimit} gives the long exact sequence
$$
\cdots \to \HH_i(L)  \to  \prod _{\ell \in \mathbb{N}} \HH_i(M^{\ell} ) \xra{\HH(\varphi)} \prod _{\ell \in \mathbb{N}} \HH_i(M^{\ell})  \to \HH_{i-1} (L) \to \cdots
$$
of homology modules. Fix an integer $i$ and note that  $\HH_i (M^{\ell}) =0$ if  $\inf (M^{\ell}) > i$.
Hence, by Condition (1) we have $\HH_i(M^{\ell}) =0$ for almost all $\ell \in \mathbb{N}$. Thus, the product $\prod _{\ell \in \mathbb{N}} \HH_i(M^{\ell} )$ is a product of finite number of $R$-modules.
Hence, $\HH(\varphi)$ is an isomorphism and therefore, $\HH_i(L)=0$ for all $i \in \mathbb{Z}$, as desired.
\end{proof}

\begin{para}{\emph{Proof of Theorem~\ref{limit2}.}}\label{para20201122c}
Since $N$ is semifree, by~\ref{para20201122b} and Condition (2)
\begin{align*}
\Hom _{\K(A)} (N ,  M)&\cong \holim \Hom _{\K(A)} (N ,  M^{\ell})\\&\cong \Hom _{\K(A)} (N , \holim M^{\ell})\\ &\cong \Hom _{\D(A)} (N , \holim M^{\ell}).
\end{align*}
Now the assertion follows from Lemma~\ref{lem20201122a}.\qed
\end{para}


\begin{para}{\emph{Proof of Theorem~\ref{HLtheorem}.}}\label{para20201122d}
It follows from Lemma~\ref{lemma for HLtheorem} that $\Ext_B ^i (N,N \otimes _B  J^{[\ell]}) \cong \Ext_B ^i (N,N \otimes _B  J)$ for all $i\geq 1$ and all $\ell\geq 1$.
Note that $\{\inf \left(N \otimes _B J ^{[\ell]}\right)\mid \ell\in\mathbb{N}\}$ is an increasing sequence of integers that diverges to $\infty$. Now, the assertion follows from Theorem~\ref{limit}.\qed
\end{para}

\section{Na\"ive liftability and proof of Theorem~\ref{thm2021017a}}\label{sec20201126c}

The notion of na\"ive liftability was introduced by the authors in~\cite{NOY} along simple free extensions of DG algebras. It is shown in~\cite[Theorem 6.8]{NOY} that along such extension $A\to A\langle X\rangle$ of DG algebras, weak liftability in the sense of~\cite[Definition 5.1]{NOY} (when $|X|$ is odd) and liftability (when $|X|$ is even) of DG modules are equivalent to na\"ive liftability. In this section, we study the na\"ive lifting property of DG modules in a more general setting using the diagonal ideal. We give the proof of Theorem~\ref{thm2021017a} in~\ref{para20201125a}.

\begin{para}\label{para20201113a}
Let $A\to B$ be a homomorphism of DG $R$-algebras such that $B$ is free as an underlying graded $A$-module. Let $(N,\partial^N)$ be a semifree DG $B$-module, and let $N |_A$ denote $N$ regarded as a DG $A$-module via $A\to B$.
Since $B$ is free as an underlying graded $A$-module, $N |_A$  is a semifree DG $A$-module.
Note that  $(N |_A \otimes _A B,\partial)$  is a DG $B$-module with
$\partial  (n \otimes b) = \partial ^N (n) \otimes b + (-1)^{|n|} n \otimes d^B (b)$ for all homogeneous elements $n \in N$ and $b \in B$.
Since $N |_A$ is a semifree DG $A$-module, $N |_A \otimes _A B$ is a semifree DG $B$-module, and we have a (right) DG $B$-module epimorphism
$\pi _N\colon N |_A \otimes _A B \to N$
defined by $\pi_N(n \otimes b)=nb$.
\end{para}

\begin{prop}\label{para20201114f}
Let $A\to B$ be a homomorphism of DG $R$-algebras such that $B$ is free as an underlying graded $A$-module. Every semifree DG $B$-module $N$ fits into the following short exact sequence of DG $B$-modules:
\begin{equation}\label{basic SES}
0 \to N \otimes _BJ \to  N |_A \otimes _A B  \xra{\pi_N} N  \to 0.
\end{equation}
\end{prop}

\begin{proof}
Since the left DG $B$-module $B^o$  is the right DG $B$-module $B$,
we have  an isomorphism $N \otimes _B B ^o \cong N$ of right DG $A$-modules such that $x\otimes b^o\mapsto xb$ for all $x\in N$ and $b\in B$.
Hence, there are isomorphisms
$N \otimes _B B^e =N \otimes _B (B^o \otimes _A B ) \cong (N \otimes _B B^o)  \otimes _A B \cong N |_A \otimes _A B$ such that $x\otimes(b_1^o\otimes b_2)\mapsto xb_1\otimes b_2$ for all $x\in N$ and $b_1,b_2\in B$.
Therefore, we get the commutative diagram
\begin{equation}\label{eq20201207s}
\xymatrix{
N \otimes _B B^e\ar[rr]^{\id_N\otimes \pi_B}\ar[d]_{\cong}&& N \otimes _B B \ar[d]^{\cong}\\
N |_A \otimes _A B\ar[rr]^{\pi_N}&&N
}
\end{equation}
of DG $B$-module homomorphisms.
Thus, by applying  $N \otimes _B - $ to the short exact sequence~\eqref{eq20201114a}, we obtain the short exact sequence~\eqref{basic SES}
in which  injectivity on the left follows from the fact  that $N$ is free as an underlying graded $B$-module.
\end{proof}

\begin{para}\label{para20201206b}
To clarify, note that the DG algebra homomorphism $\pi_B\colon B^e\to B$ defined in~\ref{para20201114d} coincides with the DG algebra homomorphism $\pi_B\colon B |_A \otimes _A B \to B$ defined in~\ref{para20201113a}. In fact, as we mentioned in the proof of Proposition~\ref{para20201114f} (the left column in~\eqref{eq20201207s} with $N=B$),
we have the isomorphism $B^e\cong B |_A \otimes _A B$.
\end{para}

We remind the reader of the definition of na\"ive liftability from the introduction.

\begin{defn}\label{defn20201125a}
Let $A\to B$ be a homomorphism of DG $R$-algebras such that $B$ is free as an underlying graded $A$-module. A semifree DG $B$-module $N$ is {\it na\"ively liftable}  to $A$ if
the map  $\pi _N$ is a split DG $B$-module epimorphism, i.e., there exists a DG $B$-module homomorphism $\rho\colon N \to  N |_A \otimes _A B$ that satisfies the equality $\pi _N \rho = \id _N$. Equivalently, $N$  is na\"ively liftable to $A$ if $\pi_N$ has a right inverse in the abelian category of right DG $B$-modules.
\end{defn}

If a semifree DG $B$-module $N$ is na\"ively liftable to $A$, then the short exact sequence~\eqref{basic SES} splits. This implies the following result.

\begin{cor}\label{naive definition}
Let $A\to B$ be a homomorphism of DG $R$-algebras such that $B$ is free as an underlying graded $A$-module. If a semifree DG $B$-module $N$ is na\"ively liftable to $A$, then $N$ is a direct summand of the DG $B$-module $N |_A \otimes _A B$ which is liftable to $A$.
\end{cor}

We use the following result in the proof of Theorem~\ref{thm2021017a}.

\begin{thm}\label{extnaive}
Let $A\to B$ be a homomorphism of DG $R$-algebras such that $B$ is free as an underlying graded $A$-module. If $N$ is a semifree DG $B$-module such that $\Ext _B ^1 (N, N \otimes _BJ)=0$, then $N$ is na\"ively liftable to $A$.
\end{thm}

\begin{proof}
Since $N$ is a semifree DG $B$-module, it follows from our Ext-vanishing assumption and~\cite[Theorem A]{NassehDG} that the short exact sequence~\eqref{basic SES} splits. Hence, $N$ is na\"ively liftable to $A$, as desired.
\end{proof}

\begin{para}{\emph{Proof of Theorem~\ref{thm2021017a}.}}\label{para20201125a}
By Theorem~\ref{Extzero} we have $\Ext_B ^i (N,N \otimes _B  J/J^{[\ell]})=0$ for all $i \geq 0$ and all $\ell\geq 1$. It follows from Theorem~\ref{HLtheorem} that $\Ext_B ^{i} (N,N \otimes _B  J)=0$ for all $i\geq 1$. Hence, by Theorem~\ref{extnaive}, $N$ is na\"{\i}vely liftable to $A$. The fact that $N$ is a direct summand of a DG $B$-module that is liftable to $A$ was already proved in Corollary~\ref{naive definition}.\qed
\end{para}

\section{Diagonal ideals in free and polynomial extensions and proof of Main Theorem}\label{sec20210107a}
This section is devoted to the properties of diagonal ideals in free and polynomial extensions of certain DG $R$-algebra $A$ with the aim to prove that such extensions are DG (quasi-)smooth over $A$; see Corollary~\ref{cor20210105a}. Then we will give the proof of our Main Theorem in~\ref{para20210107s}.
First, we focus on free extensions of DG algebras.

\begin{para}\label{para20201114a}
Let $B=A \langle X_1,\ldots,X_n \rangle$ with $n\leq \infty$, where $A$ is a divided power DG $R$-algebra. It follows from~\cite[Proposition 1.7.6]{GL} that $B$ is a divided power DG $R$-algebra.
Also, $B^o$ is a divided power DG $R$-algebra with the divided power structure $(b^o)^{(i)} = (b^{(i)})^o$, for all $b \in B$ and $i\in \mathbb{N}$.
\end{para}

The next lemma indicates that, in the setting of~\ref{para20201114a}, the map $\pi _B$ is a homomorphism of divided power algebras in the sense of~\cite[Definition 1.7.3]{GL}.

\begin{lem}\label{lem20201115a}
Let $B=A \langle X_1,\ldots,X_n \rangle$ with $n\leq \infty$, where $A$ is a divided power DG $R$-algebra. The DG algebra homomorphism $\pi _B$ preserves the divided powers.
\end{lem}

\begin{proof}
Note that $B^e$ is a divided power DG $R$-algebra by setting
$$
(b_1^o \otimes b_2)^{(i)} = \begin{cases} (b_1^i)^o \otimes b_2 ^{(i)} & \text{if $|b_1|, |b_2|$ are even and $|b_2| > 0$ } \\  0 & \text{if  $|b_1|, |b_2|$ are odd } \end{cases}
$$
for all $b_1^o \otimes b_2\in B^e$ of positive even degree and all integers $i\geq 2$.
By the properties of divided powers in~\ref{para20200329a}, for such $b_1^o \otimes b_2\in B^e$ we have
\begin{align*}
\pi_B \left( (b^o_1 \otimes b_2) ^{(i)} \right)&=\begin{cases} \pi_B\left((b_1^i)^o \otimes b_2 ^{(i)}\right) & \text{if $|b_1|, |b_2|$ are even and $|b_2| > 0$ } \\  0 & \text{if  $|b_1|, |b_2|$ are odd } \end{cases}\\
&=\begin{cases} b_1^i b_2 ^{(i)} & \text{if $|b_1|, |b_2|$ are even and $|b_2| > 0$ } \\  0 & \text{if  $|b_1|, |b_2|$ are odd } \end{cases}\\
&=(b_1 b_2) ^{(i)}\\
&=\left(\pi_B  (b^o_1 \otimes b_2)  \right)^{(i)}.
\end{align*}
Now, the assertion follows from the fact that every element of $B^e$  is a finite sum of the elements of the form
$b^o_1 \otimes b_2$.
\end{proof}

\begin{para}\label{para20201126c}
In Lemma~\ref{lem20201115a}, we assume that $A$ is a divided power DG $R$-algebra to show that $J$ is closed under taking divided powers.
Note that elements of $J$ are not all of the form of a monomial, so to define the ``powers'' of non-monomial elements we need to consider the divided powers.
For example, for a positive integer $\ell$ we cannot define $(X_i+a)^{(\ell)}$, where $a \in A$ without assuming that $A$ is a divided power DG $R$-algebra.
\end{para}

\begin{para}\label{para20201115a}
Let $B=A \langle X_1,\ldots,X_n \rangle$ with $n\leq \infty$, where $A$ is a divided power DG $R$-algebra. For $1\leq i \leq n$, the \emph{diagonal of the variable  $X_i$} is an element of $B^e$ which is defined by the formula
\begin{equation}\label{xi}
\xi _i = X_i^o \otimes 1 - 1^o \otimes X_i.
\end{equation}
Since  $\pi _B (\xi _i) =0$, we have that $\xi _i \in J$ for all $i$.
Note that if $|X_i|$ is odd, then $\xi _i^2=0$. From the basic properties of divided powers we have
\begin{equation}\label{xi(m)}
\xi _i^{(m)} = \sum _{j=0}^m (-1)^{m-j} \left(X_i^{(j)}\right)^o \otimes X_i^{(m-j)}
\end{equation}
for all $m \in \mathbb{N}$, considering the conventions that  $\xi _i ^{(0)}=1^o\otimes 1$ and $\xi_i^{(m)}=0$ for all $m\geq 2$ if $|\xi_i|=|X_i|$ is odd. Note that $\xi_i^{(1)}=\xi_i$ by definition.
Since by Lemma~\ref{lem20201115a}, the map $\pi _B$ preserves the divide powers, we see that $\xi _i^{(m)} \in J$ for all $i$ and $m\in \mathbb{N}$.
Let $\Omega=\{ \xi _i ^{(m)} \mid 1 \leq i \leq n,  \ m\in \mathbb{N}\}$.
\end{para}

\begin{lem}\label{lem20201116a}
Let $B=A \langle X_1,\ldots,X_n \rangle$ with $n\leq \infty$, where $A$ is a divided power DG $R$-algebra. The diagonal ideal $J$ is generated by $\Omega$, that is, $J= \Omega  B^e$.
\end{lem}

The statement of this lemma is equivalent to the equality $J= B \Omega  B$.

\begin{proof}
Since $\Omega \subseteq J$, we have $J' :=  \Omega  B^e\subseteq J$.
Now we show $J\subseteq J'$.

We claim that for all $1\leq i\leq n$ with $n\leq \infty$ and all $m\in \mathbb{N}$ we have the equality
\begin{equation}\label{ximb}
(X_i ^{(m)})^o \otimes 1 \equiv 1^o \otimes X_i ^{(m)} \pmod{J'}.
\end{equation}
To prove this claim, we proceed by induction on $m\in \mathbb{N}$. For the base case, since $\xi_i\in \Omega\subseteq J'$, we have $X_i^o \otimes 1 \equiv 1^o \otimes X_i \pmod{J'}$ for all $1\leq i\leq n$ with $n\leq \infty$.  Note that for all $1\leq i\leq n$ with $n\leq \infty$, we have $\sum _{j=0}^m (-1)^{m-j} (X_i^{(j)})^o \otimes X_i^{(m-j)} = \xi _i^{(m)} \equiv 0 \pmod{J'}$. Hence, we obtain a series of congruencies modulo $J'$ as follows:
\begin{eqnarray*}
(X_i ^{(m)})^o \otimes 1 + (-1)^m  (1^o \otimes X_i ^{(m)})
&\equiv& -\sum _{j=1}^{m -1} (-1)^{m-j} (X_i^{(j)})^o \otimes X_i^{(m-j)}  \\
&=& -\sum _{j=1}^{m -1} (-1)^{m-j} \left((X_i^{(j)})^o\otimes 1\right) \left(1^o\otimes X_i^{(m-j)}\right)  \\
&\equiv& -\sum _{j=1}^{m -1} (-1)^{m-j} \left(1^o\otimes X_i^{(j)}\right) \left(1^o\otimes X_i^{(m-j)}\right)  \\
&\equiv& -\sum _{j=1}^{m -1} (-1)^{m-j} (1^o \otimes X_i^{(j)} X_i^{(m-j)}) \\
&\equiv&-\sum _{j=1}^{m -1} (-1)^{m-j} \binom{m}{j} (1^o \otimes X_i^{(m)})
\end{eqnarray*}
where the third step uses the inductive hypothesis.
The claim now follows from the well-known equality $\sum _{j=1}^{m -1} (-1)^{m-j} \binom{m}{j}=-1-(-1)^m$.

Now let $\beta\in J\subseteq B^e$ be an arbitrary element. It follows from~\eqref{ximb} that there exists an element $b_{\beta}  \in B$ such that $ \beta \equiv 1 \otimes b_{\beta} \pmod{J'}$.
Since $\pi _B (\beta)=0$, we have
$\pi _B (1 \otimes b_{\beta})= b_{\beta} = 0$. Hence, $\beta \equiv 0 \pmod{J'}$, which means that $\beta\in J'$. This implies that $J \subseteq J'$, as desired.
\end{proof}

\begin{lem}\label{basis}
Let $B=A \langle X_1,\ldots,X_n \rangle$ with $n\leq \infty$, where $A$ is a divided power DG $R$-algebra. The set
$$
\Mon(\Omega)=
\begin{cases}
\{ \xi_1^{(m_1)}\cdots \xi_n^{(m_n)} \ | \ \xi_i ^{(m_i)} \in \Omega \ (1 \leq i \leq n) \} \cup \{1^o \otimes 1 \}&\text{if}\ n<\infty\\
\{ \xi_{i_1}^{(m_{i_1})}\cdots \xi_{i_t}^{(m_{i_t})} \ | \ \xi_{i_j} ^{(m_{i_j})} \in \Omega \ (i_j\in \mathbb{N}, t<\infty) \} \cup \{1^o \otimes 1 \}&\text{if}\ n=\infty
\end{cases}
$$
is a basis for the underlying graded free $B^o$-module $B^e$.
Also, the diagonal ideal $J$ is a free graded $B^o$-module with the graded basis
$\Mon (\Omega) \setminus \{1^o \otimes 1 \}$.
\end{lem}

\begin{proof}
Recall from~\ref{para20201206d} that the underlying graded $B^o$-module $B^e$ is free with the basis $\Mon(\Gamma)$. We prove the lemma for $n<\infty$; the case of $n=\infty$ is similar.

Assume that $n<\infty$. For an integer $\ell \geq 0$ let
\begin{gather*}
\Mon _{\ell} (\Gamma ) = \{ (1^o \otimes X_1^{(m_1)}X_2^{(m_2)}\cdots X_n^{(m_n)})\in \Mon(\Gamma) \mid  m_1 + \cdots + m_n = \ell \}\\
\Mon _{\ell} (\Omega ) = \{ \xi_1^{(m_1)}\xi_2^{(m_2)}\cdots \xi_n^{(m_n)}\in \Mon(\Omega)  \mid  m_1 + \cdots + m_n = \ell \}.
\end{gather*}
Also let $F_{\ell} (B^e)$ be the free $B^o$-submodule of $B^e$ generated by  $\bigcup _{0 \leq i \leq \ell} \Mon _{i} (\Gamma )$,  i.e.,
$$
F_{\ell} (B^e) =\left( \bigcup _{0 \leq i \leq \ell} \Mon _{i} (\Gamma ) \right)  B^o = \sum _{i=0} ^{\ell}  \Mon _{i}  (\Gamma) B^o
$$
Then the family $\{ F_{\ell} (B^e) \mid \ell \geq 0\}$ is a filtration of the $B^o$-module $B^e$ satisfying the following properties:
\begin{enumerate}[\rm(1)]
\item   $ B^o \otimes 1=F_{0} (B^e) \subset  F_{1} (B^e) \subset  \cdots  \subset  F_{\ell} (B^e) \subset F_{\ell +1 } (B^e) \subset \cdots \subset B^e$;
\item $\bigcup _{\ell \geq 0} F_{\ell} (B^e) = B^e$;
\item $F_{\ell} (B^e) F_{\ell '} (B^e) \subseteq F_{\ell + \ell '} (B^e)$; and
\item each $F_{\ell} (B^e) /F_{\ell -1} (B^e)$ is a free $B^o$-module with free basis  $\Mon _{\ell } (\Gamma)$.
\end{enumerate}
Regarding $B^e$ as a $B^o$-module, by~\eqref{xi(m)} for all $1\leq i\leq n$ and $m\geq 1$ we have
$$
\xi _i^{(m)} = (-1)^m(1^o \otimes X_i^{(m)}) + \sum _{j=1}^m (-1)^{m-j} \left(1^o \otimes X_i^{(m-j)}\right)\left(X_i^{(j)}\right)^o.
$$
Hence,
$\xi _i^{(m)} - (-1)^m(1^o \otimes X_i^{(m)})\in F_{m-1}(B^e)$.
Therefore, if $\xi_1^{(m_1)}\xi_2^{(m_2)}\cdots \xi_n^{(m_n)} \in \Mon _{\ell} (\Omega)$, then we get a sequence of congruencies modulo $F_{\ell -1} (B^e)$ as follows:
\begin{eqnarray*}
\xi_1^{(m_1)}\xi_2^{(m_2)}\cdots \xi_n^{(m_n)}
&\equiv & (-1)^{m_1} (1^o \otimes X_1^{(m_1)})\xi_2^{(m_2)}\cdots \xi_n^{(m_n)} \\
&\equiv& \cdots  \\
&\equiv&  (-1)^{m_1+\cdots + m_n} (1^o \otimes X_1^{(m_1)}X_2^{(m_2)}\cdots X_n^{(m_n)})\\
&=&  (-1)^{\ell} (1^o \otimes X_1^{(m_1)}X_2^{(m_2)}\cdots X_n^{(m_n)}).
\end{eqnarray*}
Thus, $\Mon _{\ell} (\Omega)$ is a basis for the $B^o$-module $F_{\ell} (B^e) /F_{\ell -1} (B^e)$.
By induction on $\ell$ we can see that $F_{\ell} (B^e)$ itself is also a free $B^o$-module with basis
$\bigcup _{0 \leq i \leq \ell} \Mon _{i} (\Omega)$.
In particular, every finite subset of $\Mon (\Omega)$ is linearly independent over $B^o$.
Since $\bigcup _{\ell \geq 0} F_{\ell} (B^e) = B^e$, the set $\Mon (\Omega)$ generates $B^e$ as a $B^o$-module.
Therefore, $\Mon (\Omega)$ is a basis of the free $B^o$-module  $B^e$.

The fact that the diagonal ideal $J$ is free over $B^o$ with the basis
$\Mon (\Omega) \setminus \{1^o \otimes 1 \}$ follows from the short exact sequence~\eqref{eq20201114a}.
\end{proof}

\begin{thm}\label{free extension}
Let $B=A \langle X_1,\ldots,X_n \rangle$ with $n\leq \infty$, where $A$ is a divided power DG $R$-algebra. The DG algebra homomorphism $B^o \to B^e$ defined by $b^o\mapsto b^o\otimes 1$ is a free extension of DG $R$-algebras, i.e.,
$B^e =B^o \langle \xi _1, \xi_2, \ldots , \xi_n \rangle$ with $n\leq \infty$.
\end{thm}

\begin{proof}
By Lemma~\ref{basis}, the set $\Mon (\Omega)$ is a basis for the underlying graded free $B^o$-module $B^e$.
To complete the proof, it suffices to show that for each $1\leq i\leq n$ with $n\leq \infty$ the element $d^{B^e} (\xi _i)$ belongs to $B^o \langle \xi _1, \xi_2, \ldots , \xi_{i-1} \rangle$ and is a cycle.
To see this, note that we have the equalities $$d^{B^e} (\xi _i) = (d^{B}(X_i))^o \otimes 1 - 1^o \otimes d^{B}(X_i)=(d^{A^{(i)}}(X_i))^o \otimes 1 - 1^o \otimes d^{A^{(i)}}(X_i)$$
in which $d^{A^{(i)}}(X_i) \in A^{(i-1)}$ is a cycle.
Applying Lemma~\ref{basis} to $A^{(i-1)}$, we see that
$(A^{(i-1)})^e$ is generated by $\{ \xi _1^{(m_1)} \ldots  \xi _{i-1}^{(m_{i-1})}\mid m_j \geq 0 \ (1\leq j \leq i-1)\}$ as an $(A^{(i-1)})^o$-module.
Since $(d^{A^{(i)}}(X_i))^o \otimes 1 - 1^o \otimes d^{A^{(i)}}(X_i) \in (A^{(i-1)})^e$ and $d^{A^{(i)}}(X_i) \in A^{(i-1)}$ is a cycle, $d^{B^e} (\xi _i)\in(A^{(i-1)})^o\langle \xi _1, \ldots , \xi _{i-1}\rangle \subseteq  B^o\langle \xi _1, \ldots , \xi _{i-1}\rangle$ and is a cycle.
\end{proof}



Our next move is to define the notion of divided powers of the diagonal ideal $J$.

\begin{para}\label{para20201116a}
Let $B=A \langle X_1,\ldots,X_n \rangle$ with $n\leq \infty$, where $A$ is a divided power DG $R$-algebra. Recall from Lemma~\ref{lem20201116a} that $J= \Omega B^e$. For an integer $\ell\geq 0$ let
$$
\Mon_{\geq \ell}(\Omega)  =
\begin{cases}
\{ \xi_1^{(m_1)}\cdots \xi_n^{(m_n)}\in \Mon(\Omega)\mid \ m_1+\cdots +m_n \geq \ell \}&\text{if}\ n<\infty\\
\{ \xi_{i_1}^{(m_{i_1})}\cdots \xi_{i_t}^{(m_{i_t})}\in \Mon(\Omega)\mid \ m_{i_1}+\cdots +m_{i_t} \geq \ell \}&\text{if}\ n=\infty.
\end{cases}
$$
which is the set of monomials that are (symbolically) products of more than or equal to $\ell$ variables.
By Lemma~\ref{basis}, the set $\Mon_{\geq 1}(\Omega)=\Mon(\Omega)\backslash \{1^o\otimes 1\}$ is a basis for $J$ as a free $B^o$-module.
We define the \emph{$\ell$-th power of $J$} to be $J ^{(\ell)} := \Mon_{\geq \ell}(\Omega)B^e$. Note that $J=J^{(1)}$ and there is a descending sequence
$$B^e \supset J \supset J ^{(2)} \supset \cdots \supset J ^{(\ell)} \supset J ^{(\ell +1)} \supset \cdots$$
of the DG ideals in $B^e$; see Lemma~\ref{lem20201116b} below.
\end{para}

\begin{lem}\label{lem20201116b}
Let $B=A \langle X_1,\ldots,X_n \rangle$ with $n\leq \infty$, where $A$ is a divided power DG $R$-algebra. For every $\ell \geq 0$ the ideal $J ^{(\ell)}$ is a DG ideal of $B^e$ and
 we have
$J J^{(\ell)} = J^{(\ell)}J \subseteq J ^{(\ell +1)}$.
Moreover, the quotient $J ^{(\ell)}/J^{(\ell +1)}$  is a DG $B$-module.
\end{lem}

\begin{proof}
We prove the assertion for $n<\infty$. The case where $n=\infty$ is treated similarly by using the appropriate notation.

By Theorem \ref{free extension}, for all $1 \leq i \leq n$ and $m_1, \ldots , m_n \geq 0$ we have
$$
\xi _i\ (\xi_1^{(m_1)}\xi_2^{(m_2)}\cdots \xi_n^{(m_n)}) = (-1)^{|\xi_i|\left(\sum_{j=1}^{i-1} m_j|\xi_j|\right)}\ (m_i+1)\ \xi_1^{(m_1)}\cdots \xi_i^{(m_i +1)}\cdots \xi_n^{(m_n)}.
$$
This shows that $J J ^{(\ell)} \subseteq J ^{(\ell +1)}$. Similarly, $J^{(\ell)}J \subseteq J ^{(\ell +1)}$.

To prove that $J ^{(\ell)}$ is a DG ideal, we show that $d^{B^e} (J ^{(\ell)}) \subseteq J ^{(\ell)}$.
Recall, from the definition, that $J$ is the kernel of the DG $R$-algebra homomorphism $\pi_B$. Since $\pi_B$ is a chain map, we have $\pi_B( d^{B^e} (J ))=d^B(\pi_B(J))=0$. Hence, $d^{B^e} (J) \subseteq J$.

Note that $d^{B^e} (\xi_i^{(m)})\in J^{(m)}$ for all $1\leq i\leq n$ and $m\geq 1$.
In fact, we have  $d^{B^e} (\xi_i^{(m)}) = \xi_i^{(m-1)} d^{B^e} (\xi_i) \in J ^{(m-1)}J \subseteq J^{(m)}$.

Now assume that $\ell \geq 2$.
If $m_1+ \cdots + m_n \geq \ell$, then we have
$$
d^{B^e}( \xi_1^{(m_1)}\xi_2^{(m_2)}\cdots \xi_n^{(m_n)})
= \sum _{i=1}^n  \pm\ d^{B^e} (\xi _i^{(m_i)}) \left(\xi_1^{(m_1)} \cdots \xi_{i-1}^{(m_{i-1})}\xi_{i+1}^{(m_{i+1})}\cdots \xi_n^{(m_n)}\right)
$$
which is an element in $\sum _{i=1}^n J^{(m_i)} J^{(\ell -m_i)}\subseteq J^{(\ell)}$. Therefore, $d^{B^e} (J ^{(\ell)}) \subseteq J ^{(\ell)}$.

The assertion that $J ^{(\ell)}/J^{(\ell +1)}$  is a DG $B$-module follows from the facts that the underlying graded $B^e$-module $J ^{(\ell)}/J^{(\ell +1)}$ is annihilated by $J$ and $B^e/J \cong B$ as graded algebras.
\end{proof}

\begin{thm}\label{Jell/Jell+1}
Let $B=A \langle X_1,\ldots,X_n \rangle$ with $n\leq \infty$, where $A$ is a divided power DG $R$-algebra. For every $\ell \geq 0$, the DG $B$-module $J^{(\ell)}/J^{(\ell +1)}$ is semifree with the semifree basis $\Mon_{\ell}(\Omega)$. In case that $n<\infty$, this is a finite semifree basis.
\end{thm}

\begin{proof}
Recall from~\ref{para20201114e} that $B^o$ is a DG $R$-subalgebra of $B^e$. By definition of $J^{(\ell)}$ from~\ref{para20201116a},
the underlying  graded $B^o$-module $J^{(\ell)}/J^{(\ell +1)}$ is free with the basis $\Mon_{\ell}(\Omega)$.
Note that the composition of the maps
$B^o \to  B^e  \xra{\pi_B} B$ defined by $b^o\mapsto b^o\otimes 1\mapsto b$ is an isomorphism, and that the $B$-module structure on  $J^{(\ell)}/J^{(\ell +1)}$ from Lemma~\ref{lem20201116b} coincides with its $B^o$-module structure.
Thus,   $J^{(\ell)}/J^{(\ell +1)}$  is free as an underlying graded $B$-module. Therefore, $J^{(\ell)}/J^{(\ell +1)}$ is a semifree DG $B$-module with semifree basis $\Mon_{\ell}(\Omega)$.
\end{proof}

\begin{para}\label{para20210105c}
Let $B=A[X_1,\ldots,X_n]$ with $n\leq \infty$ be a polynomial extension of the DG $R$-algebra $A$ with variables $X_1,\ldots,X_n$ of positive degrees. For each $1\leq i\leq n$, consider the diagonal $\xi_i$ of the variable $X_i$ defined in~\eqref{xi}. In this case we have
\begin{equation}\label{xim}
\xi_i^m=\sum _{j=0}^m (-1)^{m-j} \binom{m}{j} \left(\left(X_i^{j}\right)^o \otimes X_i^{m-j}\right).
\end{equation}
Hence, similar to~\ref{para20201115a}, we can consider the set $\{ \xi _i ^{m} \mid 1 \leq i \leq n,  \ m\in \mathbb{N}\}\subseteq J$, which we again denote by $\Omega$ in this case.
Replacing divided powers $X_i^{(m)}$ and $\xi_i^{(m)}$ by ordinary powers $X_i^{m}$ and $\xi_i^m$, we can show that Lemmas~\ref{lem20201116a} and~\ref{basis} hold in this case as well. Hence, similar to Theorem~\ref{free extension}, we have $B^e = B^o [\xi _1, \ldots , \xi _n]$. Note that in this case for an integer $\ell\geq 0$ we have $J^{\ell}=\Mon_{\geq \ell}(\Omega)B^e$ and $J^{\ell}/J^{\ell+1}$ is a semifree DG $B$-module with the semifree basis $\Mon_{\ell}(\Omega)$.
\end{para}

We can now prove the following which is a key to the proof of Main Theorem.

\begin{cor}\label{cor20210105a}
Let $n\leq \infty$. We consider the following two cases:
\begin{enumerate}[\rm(a)]
\item
$B=A[X_1,\ldots,X_n]$; or
\item
$A$ is a divided power DG $R$-algebra and $B=A \langle X_1,\ldots,X_n \rangle$.
\end{enumerate}
Then $B$ is DG quasi-smooth over $A$. If $n < \infty$, then $B$  is DG smooth over $A$.
\end{cor}

\begin{proof}
Note that by definition of $J^{(\ell)}$ and Lemma~\ref{basis}, the quotient $J/J^{(\ell)}$ is a semifree DG $B$-module with the semifree basis $\Mon(\Omega)\backslash \Mon_{\geq \ell}(\Omega)$. In case (a), set $J^{[\ell]} = J ^{(\ell)}$ and in case (b), set $J^{[\ell]} = J ^{\ell}$ for each positive integer $\ell$. The assertion follows from Lemma~\ref{lem20201116b}, Theorem~\ref{Jell/Jell+1}, and~\ref{para20210105c}.
\end{proof}

\begin{para}{\emph{Proof of Main Theorem.}}\label{para20210107s}
The assertion follows from Theorem~\ref{thm2021017a} and Corollary~\ref{cor20210105a}. \qed
\end{para}

The following result follows from Main Theorem(a) and~\ref{para20201126d}.

\begin{cor}\label{cor20201126a}
Assume that $A=R$, or $A$ is a DG $R$-algebra with $R$ containing the field of rational numbers, and let $B=A\langle X_1,\ldots,X_n\rangle$. If  $N$ is a bounded below semifree DG $B$-module such that $\Ext _B ^i (N, N)=0$  for  all $i\geq 1$, then  $N$ is na\"ively liftable to $A$. Moreover, $N$ is a direct sum of a DG $B$-module that is liftable to $A$.
\end{cor}

\section{Auslander-Reiten Conjecture and na\"ive lifting property}\label{sec20201126n}

Our study in this paper is motivated by the following long-standing conjecture posed by Auslander and Reiten which has been studied in numerous works; see for instance~\cite{AY, auslander:lawlom, avramov:svcci, avramov:edcrcvct, avramov:phcnr, huneke:voeatoscmlr, huneke:vtci, jorgensen:fpdve, nasseh:vetfp, nasseh:lrqdmi, nasseh:oeire, MR1974627, sega:stfcar}, to name a few.

\begin{ARC}[\protect{\cite[p.\ 70]{AR}}]
\emph{Let $(S,\fn)$ be a local ring and $M$ be a finitely generated $S$-module. If $\Ext^i_S(M\oplus S,M\oplus S)=0$ for all $i>0$, then $M$ is a free $S$-module.}
\end{ARC}

Our Main Theorem in this paper considers na\"ive liftability of DG modules along \emph{finite} free extensions of DG algebras. However, in dealing with the Auslander-Reiten Conjecture, we need to work with \emph{infinite} free extensions of DG algebras. So, we pose the following conjecture for which we do not have a proof yet.

\begin{NLC}\label{conj20210108b}
\emph{Assume that $A$ is a divided power DG $R$-algebra, and let $B=A\langle X_i\mid i\in \mathbb{N}\rangle$. If $N$ is a bounded below semifree DG $B$-module such that $\Ext^{i}_{B}(N\oplus B,N\oplus B)=0$ for all $i\geq 1$, then $N$ is na\"ively liftable to $A$.}
\end{NLC}

Our next result explains the relation between these conjectures.

\begin{thm}\label{thm20210108z}
If Na\"ive Lifting Conjecture holds, then the Auslander-Reiten Conjecture holds.
\end{thm}

\begin{proof}
Let $(S,\fn)$ be a local ring and $M$ be a finitely generated $S$-module with $\Ext^i_S(M\oplus S,M\oplus S)=0$ for all $i>0$. Without loss of generality we can assume that $S$ is complete in its $\frak n$-adic topology. Consider the minimal Cohen presentation $S\cong R/I$ of $S$, where $R$ is a regular local ring and $I$ is an ideal of $R$. By a construction of Tate~\cite{Tate}, there is a DG $R$-algebra
$B=R\langle X_i\mid i\in \mathbb{N}\rangle$
that resolves $S$ as an $R$-module, that is, $S\simeq B$.
The $S$-module $M$  is regarded as a DG $B$-module via the natural augmentation  $B \to S$.
This homomorphism of DG $S$-algebras induces a functor $\mathfrak{F}\colon \D(S) \to \D(B)$  of the derived categories.
Since $S\simeq B$, by Keller's Rickard Theorem~\cite{Keller}, the functor $\mathfrak{F}$ yields a triangle equivalence and its quasi-inverse is given by  $-\lotimes_BS$.

Let $N\xra{\simeq} M$ be a semifree resolution of the DG $B$-module $M$; see~\cite{avramov:ifr} for more information.
Then, as an underlying graded free $B$-module, $N$ is non-negatively graded and $\HH(N)\cong \HH(M)= M$, which is bounded and finitely generated over $\HH_0(B)\cong R$.
Note that $M$ corresponds to $N$ and $S$ corresponds to $B$ under the functor  $\mathfrak{F}$.
Since $\mathfrak{F}$ is a triangle equivalence, we conclude that
$\Ext _B^i (N\oplus B, N\oplus B)=0$ for all $i\geq 1$.
By our assumption, $N$ is na\"ively liftable to $A$. In particular, by Corollary~\ref{naive definition}, $N$ is a direct summand of $N|_R\otimes_R B$. Using the category equivalence $\mathfrak{F}$, we see that $M$ is a direct summand of  $M \lotimes  _R S$ in $\D (S)$ which is a bounded free complex over $S$,  since  $R$ is regular. Hence, $\pd_S(M)<\infty$. It then follows from \cite[Theorem 2.3]{CH} that $M$ is free over $S$.
\end{proof}

\begin{para}\label{para20210108d}
According to the proof of Theorem~\ref{thm20210108z}, we do not need to prove the Na\"ive Lifting Conjecture in its full generality for the Auslander-Reiten Conjecture; only proving it for the case where $A=R$ is a regular local ring would suffice for this purpose. Note that, despite the finite free extension case, the assumption ``$\Ext^{i}_{B}(N,N)=0$ for all $i\geq 1$'' is not enough for the Na\"ive Lifting Conjecture to be true in general. The reason is that there exist non-free finitely generated modules $M$ over a general local ring $S$ satisfying $\Ext^{i}_{S}(M,M)=0$ for all $i\geq 1$; see, for instance, \cite{JLS}.
\end{para}

\begin{para}\label{para20210108b}
In the proof of Theorem~\ref{thm20210108z}, if $S$ is resolved as an $R$-module by a finite free extension $B=R\langle X_1,\ldots,X_n\rangle$, then $S$ is known to be a complete intersection ring. Hence, in this case, our Main Theorem and Theorem~\ref{thm20210108z} just provide another proof for the well-known fact that complete intersection rings satisfy the Auslander-Reiten Conjecture; see, for instance, \cite{AY,avramov:svcci,jorgensen:fpdve}.
\end{para}



\providecommand{\bysame}{\leavevmode\hbox to3em{\hrulefill}\thinspace}
\providecommand{\MR}{\relax\ifhmode\unskip\space\fi MR }
\providecommand{\MRhref}[2]{%
  \href{http://www.ams.org/mathscinet-getitem?mr=#1}{#2}
}
\providecommand{\href}[2]{#2}

\end{document}